\documentclass[12pt]{amsart}

\usepackage{epsf,amssymb,times,overpic,amscd}

\newtheorem{thm}{Theorem}[section]
\newtheorem{prop}[thm]{Proposition}
\newtheorem{lemma}[thm]{Lemma}
\newtheorem{cor}[thm]{Corollary}

\newtheorem{question}[thm]{Question}

\theoremstyle{definition}
\newtheorem{defn}[thm]{Definition}

\theoremstyle{remark}
\newtheorem{q}[thm]{Question}
\newtheorem{rmk}[thm]{Remark}
\newtheorem{example}[thm]{Example}

\newcommand{\C}{\mathbb{C}}

\newcommand{\R}{\mathbb{R}}
\newcommand{\Z}{\mathbb{Z}}

\newcommand{\E}{\mathcal{E}}

\newcommand{\PP}{\mathcal{P}}
\newcommand{\DD}{\mathcal{D}}
\newcommand{\WW}{\mathcal{W}}
\newcommand{\EE}{\mathcal{E}}
\newcommand{\lra}{\longrightarrow}

\newcommand{\lmt}{\longmapsto}
\newcommand{\tw}{\mathrm{tw}}

\DeclareMathAccent{\ring}{\mathalpha}{operators}{"17}

\newcommand{\be}{\begin{enumerate}}
\newcommand{\ee}{\end{enumerate}}

\begin{document}

\title[]{Notes on open book decompositions for Engel structures}

\author{Vincent Colin, Francisco Presas, Thomas Vogel}
\address{V. Colin, Universit\'e de Nantes, 2 rue de la Houssini\`ere, 44322 Nantes, France }
\email{vincent.colin@univ-nantes.fr}

\address{F. Presas, ICMAT, Calle Nicol{\'a}s Cabrera, no. 13-15, Campus de Cantoblanco UAM, 28049 Madrid Spain}
\email{fpresas@icmat.es}

\address{T. Vogel, Math. Institut der Ludwig-Maximilians-Universit{\"a}t, Theresienstr. 39, 80333 M{\"u}nchen, Germany}
\email{tvogel@math.lmu.de}

\date{This version: \today}



\thanks{V. Colin thanks the ERC Geodycon and the ANR Quantact for their support. F. Presas is supported by the Spanish Research Projects SEV-2015-0554, MTM2016-79400-P and MTM2015-72876-EXP}

\begin{abstract}
We relate open book decompositions of a $4$-manifold $M$ with its Engel structures. Our main result is, given an open book decomposition of $M$ whose binding is a collection of $2$-tori and whose monodromy preserves a framing of a page, the construction of an Engel structure whose isotropic foliation is transverse to the interior of the pages and tangent to the binding. In particular the pages are contact manifolds and the monodromy is a contactomorphism.  As a consequence, on a parallelizable closed $4$-manifold, every open book with toric binding carries in the previous sense an Engel structure. Moreover, we show that amongst the supported Engel structures we construct, there is a class of loose Engel structures.

\end{abstract}

\maketitle
\section{Introduction}
One particularly fruitful approach to the study of contact structures is the use of open book decompositions, that was introduced by E.~Giroux \cite{giroux}. He showed that every cooriented contact structure on a closed $(2n-1)$-manifold  is supported by an open book decomposition of the underlying manifold, i.e. it can be defined by a $1$-form $\alpha$ such that $\alpha$ gives a contact structure on the binding and $\alpha$ turns the pages into Weinstein domains. If $n=2$, then there is a one-to-one correspondence between open book decompositions of $M$  up to positive stabilization and isotopy classes of contact structures on $M$. 

The simpler part of this correspondence, the construction of a contact structure starting from an open book decomposition, was established by W.~Thurston and H.~Winkelnkemper \cite{tw}. 
 
Engel structures form a class of plane fields on $4$-manifolds which have attracted some interest recently, although our understanding of their properties is relatively underdeveloped.  We refer to \cite{montgomery,thomas} for explanations concerning the motivation of their study. Here we note that Engel structures  are closely related to contact structures in dimension $3$. This fact was applied successfully in \cite{thomas} to prove that every parallelizable $4$-manifold admits an Engel structure. 

The purpose of this note is to prove an analogue of the Thurston-Winkeln\-kemper theorem for Engel structures: Under the natural  assumptions that the binding is a collection of $2$-tori and that the monodromy preserves a framing of a page, we will obtain Engel structures starting from an open book decomposition of the manifold that make the pages contact, see Theorem~\ref{t:existence}. This follows from Corollary~\ref{cor: contact} for every open book decomposition with toric binding of a parallelizable $4$-manifold.

We also address stabilization constructions and the uniqueness question. In Theorem \ref{t:loose unique} we show that amongst the supported Engel structures we construct there is a special class of loose ones which is invariant under stabilization. We also raise the problem of studying Engel structures supported by an open book with tight contact pages, a property that is invariant under stabilization by Theorem~\ref{tight}.

Whether an Engel structure is homotopic --~through Engel structures~-- to an Engel structure which is adapted to an open book decomposition is unknown. Notice here that it is also unknown whether non-loose Engel structures exist. Finding a supporting open book decomposition for a loose Engel structure up to Engel homotopy amounts to, thanks to  our construction, finding an open book decomposition and a framing $(e_1,e_2,e_3,e_4)$ such that the pages are transverse to  a vector field $e_1'$ homotopic to $e_1$ and the binding tangent to $e_1'$ and linearly foliated by $e_1'$.

Also, our definition of an Engel structure being adapted to an open book decomposition is not tightly suited  since the requirements in the definition only use the associated even contact structure.

\section{Engel and even contact structures}
We assume that the reader is familiar with the basics of contact topology in dimension $3$, a lot of the material we use is covered in \cite{geiges-book}. We now give definitions of Engel structures and associated distributions. 

\begin{defn} Let $M$ be a  $2n$-dimensional manifold. A (co--oriented) even contact structure on $M$ is a hyperplane field $\E$ defined as the kernel of a $1$-form $\alpha$ so that $\alpha \wedge (d\alpha )^{n-1}$ never vanishes.
\end{defn}
To each even contact structure one can associate a line field $\WW \subset \E$, called its {\it isotropic foliation} or {\it kernel foliation},
which is the kernel of $d\alpha$ restricted to $\E$. 
If $n=2$ then the even contact condition is equivalent to $[\E,\E]  = TM$.

\begin{defn}
An Engel structure $\DD$ is a smooth plane field on a $4$-manifold $M$ such  
that  $[\DD,\DD]  =\E$ is an even contact structure.
\end{defn}

Notice that if an even contact structure $\E$ is induced by an Engel structure, i.e. $[\DD,\DD]=\E$, then $\WW$ is tangent to $\DD$. Thus, an Engel structure $\DD$ induces a flag of distributions 
$$
\WW\subset\DD\subset\E\subset TM.
$$
Moreover, if $\E$ comes from an Engel structure, then it is canonically oriented, and an orientation of $TM$ induces an orientation of $\WW$ and vice versa. If we assume that both $\DD$ and $M$ are oriented, then this shows that $\DD$ induces a framing of $TM$ which is well-defined up to homotopy. In particular, two Engel structures cannot be homotopic through Engel structures if the associated framings are not homotopic. Quite recently it was shown in \cite{cppp} that every  framing of a parallelizable $4$-manifold is homotopic to the framing induced by an Engel structure on $M$.

\begin{example}
On $\R^4$ the $1$-form $dz-ydx$ defines an even contact structure. The isotropic foliation is spanned by $\partial_w$ (where $w$ is the fourth coordinate). Locally, every even contact structure is diffeomorphic to the one given here.
\end{example}

The fact that there is a line field associated to an even contact structure/Engel structure means that Gray's stability theorem cannot hold for these distributions: A varying family of even contact structures or Engel structures will induce a varying family of isotropic foliations. Hence, the even contact structures/Engel structures in the family are not diffeomorphic to each other since dynamical properties of the isotropic foliation can change. This is explored in \cite{montgomery}.

If $\EE$ is an even contact structure on $M$ and if $Y$ is a hypersurface in $M$ transverse to $\WW$, then the plane field $$\xi =TY \cap \E$$ is a contact structure on $Y$.

Recall also that even contact structures satisfy the h-principle: A formal (oriented) even contact structure on $M$ is a pair $(\EE,\WW)$ where $\EE$ is an oriented hyperplane field on $M$ and $\WW\subset \EE$ is an oriented line field. According to \cite{mcduff}, the space of even contact structures is homotopy equivalent to the space of formal even contact structures. 

\begin{rmk}
Once  the even contact structure is fixed, the homotopy class of $\DD$ is also fixed as a plane field in $\E$. 
Indeed, let $W$ be a vector field generating $\WW$ 
and $N$ be a vector field orthogonal to $\DD$ in $\E$ (assuming we have fixed a metric).
The path of vector fields $\cos s W +\sin s N$, $s\in [0,\pi /2]$ interpolates between $W$ and $N$, 
and thus gives a path between their normal plane fields in $\E$.

However, for a fixed pair $\WW \subset \E$, the homotopy class of $\DD$ as a distribution sitting on the sequence $\WW \subset \DD \subset \E$ is not fixed and it corresponds to the homotopy classes of sections of the $S^1$-bundle $S(\E /\WW) \to M$.
\end{rmk}

The main motivating example of an even contact structure is the preimage of the contact structure under the projection $M=Y\times[0,1]\lra Y$ for a contact structure $\xi$ on $Y$. The fibers of the projection are the leaves of the isotropic foliation. The following lemma is a slight modification of this example.

\begin{lemma}  \label{l: e product}
Let $(Y,\xi)$ be a contact manifold and $\varphi$ the coordinate on $[0,1]$. Then the  hyperplane field $\EE =\xi \oplus \R (\partial_\varphi +L)$ on $Y\times[0,1]$ is an even contact structure whose isotropic foliation is spanned by $\partial_\varphi+L$ if and only if $L$ is a contact vector field on $(Y,\xi)$. 
\end{lemma}

In order to find an Engel structure $\DD$ such that the even contact structure $\EE$ using Lemma~\ref{l: e product} satisfies $\EE=[\DD,\DD]$ we require in addition, that $\xi$ admits a trivialization $C_1,C_2$.

\begin{lemma}
Let $L$ be a contact vector field on $(Y,\xi)$. For a sufficiently big positive integer $k$ the plane field spanned by
$$
W = \partial_\varphi+L \textrm{ and } X_k = \cos(k\varphi)C_1+\sin(k\varphi) C_2
$$
is an Engel structure $\DD_k$ such that $[\DD_k,\DD_k]=\EE$. In particular, the isotropic foliation is spanned by $W$.   
\end{lemma}

\section{Open book decomposition of a $4$-manifold}

\subsection{Open books with torus binding}

In this section we summarize basic definitions concerning open book decompositions. For background on open book decompositions we refer to \cite{etnyre} when the underlying manifold has dimension $3$, to \cite{giroux, colin} for the odd-dimensional case and to \cite{quinn} for the general case. 
\begin{defn}
Let $M$ be a closed manifold. An {\em open book decomposition} of $M$ is a pair $(K,\theta)$ where
\begin{itemize}
\item $K\subset M$ is a non-empty, not necessarily connected, submanifold of codimension two with trivial normal bundle,
\item $\theta :M\setminus K \rightarrow S^1$ is a fibration, and 
\item one can choose coordinates on a neighborhood
\begin{align*}
N(K) & \simeq K\times D^2 \\
& =\{ (x,(r, \phi ))\,|\,x\in K \textrm{ and } (r,\phi)\in D^2 \textrm{ are polar coordinates}\}
\end{align*}
of $K$ such that $\theta(x,(r,\phi))=\phi$ for $r\neq 0$.
\end{itemize}
\end{defn}

The submanifold $K$ is called the {\em binding} of the open book decomposition. 
For $t \in S^1$, the preimage $\theta^{-1} (t)$ will be called a {\it fiber}
of $(K,\theta )$. The compactification of any fiber obtained by addition of $K$ is called a {\it page}. The natural orientation of $S^1$ lifts through $\theta$ to a natural coorientation of the pages. If the ambient manifold $M$ is oriented, then we obtain a natural orientation on 
the pages, and $K$ is naturally oriented as the boundary of a page.

The simplest example of a manifold with an open book decomposition is the $2$-sphere. In this case, the binding consists of a pair of points, the complementary annulus fibers over the circle and the fibers are intervals.  

When $M$ is an oriented $3$-manifold, then $M$ admits an open book decomposition according to a theorem of Alexander \cite{alexander}.
\begin{defn}(Giroux \cite{giroux})
Let $(M,\xi)$ be a contact $3$-manifold and $(K,\theta)$ an open book decomposition of $M$. Then $(K,\theta)$ {\em supports} $\xi$ if there is a defining contact form $\alpha$ for $\xi$ such that 
\begin{itemize}
\item[(i)] $\alpha|_K$ is positive, and
\item[(ii)] the restriction of $d\alpha$ to fibers of $(K,\theta)$ is a positive area form. 
\end{itemize}
\end{defn}
According to \cite{giroux}, in any odd dimension, every contact structure has a supporting open book decomposition. See \cite{presas} for a complete written account of Giroux' proof.

When $M$ has dimension four, $K$ is a disjoint union of embedded closed surfaces. By the Poincar{\'e}-Hopf-index theorem the Euler characteristic of $K$ equals the Euler characteristic of $M$. For $4$-manifolds we will require in addition that $K$ is a union of tori.  Then $\chi(M)=0$ , and this is a condition Engel manifolds satisfy. 

Moreover, it follows from Novikov's theorem that the signature of a $4$-manifold admitting an open book decomposition vanishes. However, $4$-manifolds which admit an open book decomposition do not always admit a Spin-structure. Consider for example the non-trivial $S^2$-bundle over $S^2$.

It seems to be unknown  whether or not all oriented $4$-manifolds with vanishing signature admit an open book decomposition \cite{quinn}. Below we describe simple examples of open book decompositions on $4$-manifolds which fiber over $3$-manifolds. 

So far we have considered open book decomposition from an intrinsic point of view. From a more extrinsic point of view, an open book decomposition of an $n$-manifold $M$ is a triple $(Y,h,\phi)$ where 
\begin{itemize}
\item $Y$ is a compact oriented $(n-1)$-manifold with non-empty boundary,
\item $h:Y\rightarrow Y$ is a diffeomorphism which is the identity on a neighborhood of $\partial Y$, and
\item $\phi : \Sigma_0 (Y,h) \rightarrow M$ is a homeomorphism where $\Sigma_0 (Y,h)$ is the relative mapping torus
of $(Y,h)$. 
\end{itemize}
This determines the manifold only up to homeomorphism. Recall that the relative mapping torus of $(Y,h)$ is
$$
\Sigma_0 (Y,h) = Y\times [0,1] /\sim_h ,
$$
where $\sim_h$ is the equivalence relation 
\begin{align*}
(h(x),0)\sim_h (x,1)&\textrm{ for all }x\in Y\textrm {, and} \\ 
(x,t)\sim_h (x,t')&\textrm{ for all } x\in \partial Y, t,t' \in [0,1].
\end{align*}

The triple  $(Y,h,\phi )$ gives rise to an open book decomposition of $M$ whose binding is $\phi (\partial Y \times [0,1] /\sim_h )$ and whose pages are $Y_t =\phi (Y \times \{ t\} )$. Conversely, any open book decomposition can be seen as the relative mapping torus of some $(Y,h)$ (notice that we usually do not mention the identification $\phi$).

\begin{example}\label{ex: product}
Let $(K_0 ,\theta_0 )$ be an open book decomposition of a $3$-manifold $Y$, also given by the relative mapping torus of $(S_0 ,h_0 )$. Here $S_0$ is a compact oriented surface with non-empty boundary.
Then $(S^1 \times K_0 ,id\times \theta_0 )$ is an open book decomposition of $S^1 \times Y$, also described by the relative mapping torus of $(S^1 \times S_0 ,id\times h_0 )$.
\end{example}

We close this section with an example of an  open book decomposition of $S^3$. This example will appear in the discussion about stabilizations of open book decompositions. 

\begin{example} \label{ex:S3 open book}
On $S^3\subset\mathbb{C}^2$ we consider the contact form $\alpha=r_1^2d\varphi_1+r_2^2d\varphi_2$. Here $r_1,\varphi_1$, respectively $r_2,\varphi_2$, are polar coordinates on the first, respectively second factor of $\C^2$.

Then the Hopf link $K_0=\{r_1=0\}\cup\{r_2=0\}$ is the binding of an open book decomposition which supports $\mathrm{ker}(\alpha)$. The complement of the Hopf link $K_0$ fibers over  $S^1$ as follows
\begin{align*}
\theta_0 : S^3\setminus K &\longrightarrow S^1 \\
((r_1,\varphi_1),(r_2,\varphi_2)) & \longmapsto \varphi_1+\varphi_2.
\end{align*}
One can check that $(K_0,\theta_0)$ is an open book decomposition of $S^3$ and that $(K_0,\theta_0)$ supports $\xi= \mathrm{ker}(\alpha)$. The page $A_0=\theta^{-1}_0(0)$ of the open book is the annulus 
$$
\{ (r,\varphi,\sqrt{1-r^2},-\varphi)\in S^3 \,|\, 0<r<1 \textrm{ and }\varphi\in[0,2\pi]\}
$$ 
and the monodromy of this open book decomposition is a positive Dehn twist along the circle $\{r=1/2\}\subset\theta^{-1}_0(0)$.
\end{example}

\subsection{Stabilization} \label{stabil}
As we have explained in the introduction, isotopy classes of contact structures on $3$-manifolds  are in one-to-one correspondence with open book decompositions up to positive stabilization. In this section we suggest a definition of stabilization for open book decompositions of  $4$-manifolds. 

Since we propose to study the relationship between Engel structures and open book decompositions, the property that the binding is a union of tori should be preserved. 

Let $M$ be equipped with an open book decomposition $(K,\theta)$ obtained as the relative mapping torus of $(Y,h)$. We pick a properly embedded annulus $(A,\partial A )\subset (Y,\partial Y)$ such that no   component of $\partial A$ is contractible in $\partial Y$.

We then consider the $3$-manifold $Y'$ obtained by gluing $S^1 \times [0,1] \times [-1,1]$ to $Y$, where the gluing map takes $S^1 \times \{ 0,1\} \times [-1, 1]$ to $\partial Y$ so that the two curves $S^1 \times \{ 0,1\} \times \{ 0\}$ are identified with $\partial A$ and $T=A\cup (S^1 \times [ 0,1] \times \{ 0\})$ is a torus. The boundary of $Y'$ is still a union of tori.

Let $\gamma\subset T$ be an embedded closed curve intersecting the circle $S^1 \times \{ \frac{1}{2} \} \times \{ 0\}$ exactly once. We then take a small tubular neighborhood $N(T)$ of $T$ in $Y'$ and consider the fibered Dehn twist along $T$ in the direction of $\gamma$, supported in $N(T)$. We call it $\tau_{\gamma}$.
and consider the  diffeomorphism $h'=\tau_{\gamma} \circ h$ of $Y'$.

\begin{prop} 
The relative mapping torus of $(Y',h')$ is diffeomorphic to $M$ by a diffeomorphism which induces an injection from a page $Y$ of $\Sigma_0 (Y,h)$ to a page $Y'$ of $\Sigma_0 (Y',h')$.
\end{prop}

\begin{proof} Recall that an elementary positive stabilization on a $3$-dimensional open book is obtained by doing the connected sum with $S^3$ on a suitable  neighborhood $N(a)$ of a properly embedded arc $a$ \cite{giroux, giroux-goodman}. 

We start from the open book decomposition $(K_0 ,\theta_0 )$ of $S^3$ given by the relative mapping torus of a positive Dehn twist on an annulus as described in Example~\ref{ex:S3 open book}. Its binding is a positive Hopf link. Let $A_0$ be a page of this open book and
$a$ be a properly embedded arc on $A_0$ going from one boundary component to the other.

Next, we consider the manifold $S^1 \times S^3$ together with the open book 
$(S^1 \times K_0,\mathrm{id}\times\theta_0 )$. It contains the annulus $S^1 \times a$ and the associated neighborhood $S^1 \times N(a)$ where $N(a)$ is a neighbourhood of $a$  as in \cite{giroux}. The complement of $S^1 \times N(a)$  in $S^1 \times S^3$ is diffeomorphic to $S^1 \times B^3$.

Now, we take coordinates $S^1 \times [0,1]$ on $A \subset Y$ so that 
$\gamma = \{ 0\} \times [0,1] \subset A$. As in \cite{giroux}, we can find a neighbourhood $N(A)$ of $A$ in $M$ with coordinates $S^1 \times B^3$, so that the pages of $(K,\theta )$ define a ``partial'' open book decomposition on  $S^1 \times B^3$. This open book is conjugate to $S^1 \times (N(a), K_0 ,\theta_0 )$.

Since both restrictions $(S^1 \times B^3 ,K, \theta )$ and $S^1 \times (N(a), K_0 ,\theta_0 )$ are $S^1$-invariant, we can form an $S^1$-invariant connected sum of open books, i.e.
glue $(S^1 \times S^3 \setminus int (S^1 \times N(a)), S^1 \times K_0 ,id\times \theta_0 )$ to $(M\setminus int N(A) ,K,\theta ))$ in an $S^1$-invariant way to get a manifold obviously diffeomorphic to $M$, with an open book of page $Y'$ and monodromy $h'$.
\end{proof}
Another way to argue in the previous proof would be as follows: The stabilized manifold is obtained from the original manifold by removing a copy of $S^1\times D^3$ and gluing it back in. According to \cite{laudenbach-poenaru} every diffeomorphism of $S^1\times \partial D^3$ extends to $S^1\times D^3$. Hence, the diffeomorphism type of the manifold does not change when this surgery operation is applied. 

\section{Open book decompositions and even contact structures} \label{s:even open}

Let $M$ be an oriented closed $4$-manifold together with an open book decomposition $(K,\theta )$.
\begin{defn} \label{d:adapted even}
An even contact structure $\EE$ with kernel $\WW$ is {\em adapted to} (or {\em supported by}) $(K,\theta)$, if
\begin{itemize}
\item $K$ is a union of tori.
\item $\WW$ is transverse to the fibers of $\theta$. In particular, the interior of the pages $Y_\theta$ 
are naturally contact manifolds for the
contact plane $\xi_\theta =\EE \cap TY_\theta$. 
We additionally require that $\xi_\theta$ is a positive contact
structure for the canonical orientation of $Y_\theta$.
\item $\EE$ is transverse to $K$.
\item $\WW$ restricts to a linear vector field on each connected component of $K$. 
\item For each $\theta \in S^1$ there is a collar $U\simeq \partial Y_\theta\times(-1,0]$ of $\partial Y_\theta$ in $Y_\theta$ such that $K=\partial Y_\theta\times\{0\}$ and the characteristic foliation on each connected component of $\partial Y_\theta\times\{r\}, r\in(-1,0],$ is linear. 
\end{itemize}
\end{defn}

\begin{rmk} \label{rem:uni} Since all linear vector fields on $T^2$ are homotopic, 
two even contact structures supported by an open book are homotopic through even contact structures if the contact structures in the pages are homotopic as $1$-parametric families of plane fields: 
The associated kernel foliations are homotopic through foliations 
which are transverse to the interiors of the pages and induce linear vector fields on the binding.

Two even contact structures on $M$ supported by the same open book
$(K,\theta )$ are not necesarily homotopic as formal even contact structures. Out of the loop of contact structures, we construct a gluing contactomorphism, by Gray's stability. A sufficient condition for two different even contact structures to be homotopic through formal even contact structures is that the associated gluing contactomorphisms are isotopic through formal contactomorphisms. In particular, if the return maps are contact isotopic, then the even contact structures are formally homotopic and therefore homotopic  through even contact structures according to \cite{mcduff}.

Note that the uniqueness holds in the $3$--dimesional contact case, because the return map is an exact symplectomorphism of a surface and two symplectomorphisms of a surface are isotopic through symplectomorphims if they are isotopic through diffeomorphisms.
\end{rmk}

\begin{example} \label{ex:open book even}
Let $\xi$ be a contact structure on $Y$ carried by an open book $(K,\theta)$, and $\alpha$ an adapted contact form. Its Reeb vector field $X$  is transverse to the pages and tangent to the binding. We assume that $\xi$ is trivial as a vector bundle.

Let $M=\widetilde{\mathbb{P}\xi}$ be the manifold consisting of $1$-dimensional oriented subspaces of the contact planes and 
$$
\mathrm{pr} : \widetilde{\mathbb{P}\xi} \longrightarrow M
$$
the $S^1$-fibration which sends a line in $\xi_p$ to $p$. Then $\mathrm{pr}^*\alpha$ defines an orientable even contact structure $\mathcal{E}$ whose isotropic foliation is tangent to the fibers of $\mathrm{pr}$.

 We want to modify $\mathcal{E}$ so that the open book decomposition $(\mathrm{pr}^{-1}(K),\theta\circ\mathrm{pr})$ of $\widetilde{\mathbb{P}\xi}$ is adapted to the resulting even contact structure, that will be denoted  $\EE_\varepsilon$. For this we identify $S^1\times Y$ with $\widetilde{\mathbb{P}\xi}$, denote $t$ the coordinate in $S^1$ and let $\widetilde{X}$ be the lift of the Reeb field $X$ to $S^1 \times Y$ using the connection $ \{ 0 \}\oplus TY$.   For $\varepsilon>0$ small enough the $1$-form
$$
\alpha_\varepsilon := \mathrm{pr}^*\alpha-\varepsilon dt
$$
still defines an even contact structure $\EE_\varepsilon$. Let $\partial_t$ denote the vector field tangent to the fibers dual to $dt$.  Then $\partial_t+\varepsilon \widetilde{X}$ is tangent to $\EE_\varepsilon$ and this vector field preserves $\EE_\varepsilon$. The isotropic foliation of $\mathcal{E}_\varepsilon$ is therefore spanned by $\partial_t+\varepsilon \widetilde{X}$ and the open book $(\mathrm{pr}^{-1}(K),\theta\circ\mathrm{pr})$ supports $\EE_\varepsilon$ provided that $\varepsilon>0$.   
\end{example}

When one views an open book decomposition supporting an even contact structure as a mapping torus, then the monodromy is a contact diffeomorphism of a page. Using Eliashberg's classification of overtwisted contact structures \cite{overtwisted} it is easy to obtain contact diffeomorphisms from diffeomorphisms which preserve a given plane field up to homotopy. 
\begin{prop} \label{p:fromframings to contact diffeos}
Let $Y$ be a closed $3$-manifold, $e= (e_0 ,e_1 ,e_2 )$ be a parallelization
of $TY$ and $h$ be a diffeomorphism of $Y$ such that  $h_* (e )$ is homotopic to $e$. Then there exists a contact structure $\xi$ on $Y$, together with a contact diffeomorphism $\psi$ of $(Y,\xi )$ isotopic to $h$ such that $\xi$ contains a non vanishing vector field $X$ with $\psi_* X$ homotopic to $X$ in $\xi$.
\end{prop}

\begin{proof} 
Denote $e_i^0=e_i$, $i=0,1,2$; $e_i^1=h_*(e_i)$, $i=0,1,2$. By hypothesis, there exists a family $(e_0^t, e_1^t, e_2^t)$ interpolating between the two previous framings. Denote $\xi_t= \langle e_0^t, e_1^t \rangle$. By the contractibility of the ball we may assume that after homotopy of the framing $e_i$ and isotopy of the diffeomorphism $h$, there exists a ball $U \subset Y$  such that $\xi_{t|U}= \xi_{0|U}$ is contact and $\xi_{0|U}$ contains an overtwisted disk. By \cite{overtwisted}, there exists a continuous family $\xi_{t,s}$ such that $\xi_{t,0}= \xi_t$ and $\xi_{t,1}$ are contact structures. We want to be careful by applying \cite{overtwisted} as follows:
\begin{enumerate}
\item first for $t=0$, finding $\xi_{0,s}$.
\item then we have a solution for $t=1$, given by $\xi_{1,s}= h_* \xi_{0,s}$
\item the relative character with respect to the parameter of \cite{overtwisted} allows us to extend it to $\xi_{t,s}$.
\end{enumerate}
By contractibility of the interval, there is an extension $(e_0^{t,s}, e_1^{t,s})$ that coincides with $(e_0^t, e_1^t)$ for $s=0$. Moreover, $\xi_{t,s} = \langle e_0^{t,s}, e_1^{t,s} \rangle$, $h_*e_0^{0,1}= e_0^{1,1}$ and $h_*e_1^{0,1}= e_1^{1,1}$.

By Gray's stability, there exists a flow $\varphi_t: Y \to Y$ such that $(\varphi_t)_* \xi_{0,1}= \xi_{t,1}$. Therefore $(\varphi_1^{-1} \circ h)_* \xi_{0,1} = \xi_{0,1}$ shows that $(\varphi_1^{-1} \circ h)$ is the required contactomorphism . Finally, $X=e_0^{0,1}$ satisfies the stated hypothesis.

\end{proof}

The same proof works for compact manifolds with boundary when the
diffeomorphism $h$ is the identity on the boundary and $h_*(e)$ is homotopic to $e$ relative to the boundary. We then obtain a contact diffeomorphism $\psi$ supported in the interior of $Y$ and a nowhere vanishing Legendrian vector field $X$ such that $\psi_*(X)$ is homotopic (relative to the boundary) to $X$ through nowhere vanishing Legendrian vector fields.  

The assumption in Proposition~\ref{p:fromframings to contact diffeos} that $h$ preserves a framing up to homotopy turns out to be not too restrictive by an observation of H.~Geiges \cite{geiges rev}.

\begin{lemma} \label{l:geiges trick}
Let $M$ be a closed oriented $4$-manifold with trivial tangent bundle and $X$ a nowhere vanishing vector field on $M$. Then there are vector fields $e_1,e_2,e_3$ such that   $X,e_1,e_2,e_3$ is a framing. 
\end{lemma}

\begin{proof}
We identify $TM\simeq M\times \R^4$ with $M\times\mathbb{H}$. Then $X$ together with $e_1=iX$, $e_2=jX$ and $e_3=kX$ is a framing of $TM$. 
\end{proof}
In particular, we can apply this lemma to the suspension vector field on the mapping torus of a diffeomorphism $h: N\lra N$ when the mapping torus is parallelizable. Projecting $e_1,e_2,e_3$ to the fibers of a mapping torus we obtain a framing of $N$ with the desired properties.

\begin{thm} \label{t: suspension} 
Let $Y$ be a compact $3$-manifold bounded by a non-empty union of tori, and 
$\xi$ a positive contact structure on $Y$ which prints a linear characteristic linear
foliation on $\partial Y$. 
If $h$ is  a  contact diffeomorphism of $(Y,\xi )$ which is the identity near the boundary, 
then the relative mapping torus of $(Y,h)$ carries an even contact structure which induces $\xi$ on (a small retraction of) the
pages.  This even contact structure is uniquely determined up to homotopy amongst even
contact structures carried by the triple $(Y,\xi , h)$.
\end{thm}

\begin{proof} 
For simplicity, we assume that $\partial Y$ is connected. Let $Y_s=Y\times[0,1]/\sim$ be the mapping torus of $(Y,h)$. We denote the suspension coordinate on $Y_s$ by $\varphi$ and fix a coordinate system $(x,y)$ on $T^2\simeq\partial Y$ such that $\partial Y$ is linear in terms of the coordinates. The suspension vector field is $\partial_\varphi$ and $D^2_{r_0}\subset\R^2$ is a disc of radius $r_0$. 

We view  $M$ as the mapping torus $Y_s= Y\times [0,1] /\sim$ of $(Y,h)$,with a neighborhood $N(K)=\partial Y \times D^2_{r_0}$ of the binding $K$ of the binding  attached along the boundary such that $\{y\}\times\partial D^2_{r_0}$ is mapped to  $\{y\}\times S^1$  in $\partial Y_s=\partial Y\times S^1$ for all $y\in\partial Y$.

Because $h$ preserves $\xi$ we get a well-defined $2$-plane field $\xi_s$ on $Y_s$ which is tangent to the fibers of $\theta : Y_s\lra S^1$. Together with the suspension vector field $\partial_\varphi$ this plane field spans an even contact structure $\EE_s$ on $Y_s$ whose isotropic foliation is directed  by $\partial_\varphi$. 

Since $h$ is the identity on a neighborhood of $\partial Y$, and the characteristic foliation $\xi(\partial Y)$ is linear we can choose a collar $U\simeq \partial Y\times[0,2\varepsilon),\varepsilon>0,$ such that $h|_U=\mathrm{id}_U$ and 
\begin{equation*}  
\xi=\ker(\cos(r+a)dx-\sin(r+a)dy).
\end{equation*}
where the parameter $a\in \R$ is determined by the slope of the linear characteristic foliation $\xi(\partial Y)$. Since we are free to choose the coordinates $(x,y)$, we may assume that the slope is positive in terms of these coordinates. On $U\times S^1\subset Y_s$ we consider the even contact structure  $\EE_0=\xi\oplus\R\partial_\varphi$. The isotropic foliation is of course spanned by $\partial_\varphi$.  We want to homotope $\EE_0$ through even contact structures such that isotropic foliation of the resulting even contact structure $\EE_1$ is $\partial_\varphi+\partial_y$.

Recall that the space of contact vector fields on a given contact manifold is in one-to-one correspondence to the space of smooth functions once a contact form is fixed (cf. \cite{mcduffsalamon} Section I.3.4). This implies the existence of  a contact vector field $L$  with support in $U$ and $L=\partial_y$ on a collar  $V=\partial Y\times[0,\varepsilon]\subset U$ of $\partial Y$.

We use the same notation for $L$ and its lift to the product $Y\times S^1$. The desired homotopy of even contact structures is  $\EE_s=\xi_s\oplus\R(\partial\varphi+sL)$ with $s\in[0,1]$.

Finally, we extend $\EE_s$ to an even contact structure over $V\cup N(K)$: Let $(r,\varphi)$ be polar coordinates on the second factor of  $N(K)=\partial Y_s\times \R^2$ with $r\in [0,1]$ and  
$$
\EE_b=\ker(dx+r^2d\varphi-r^2dy).
$$
One can check by a simple computation that $\EE_b$ is an even contact structure whose isotropic foliation is spanned by $\partial_\varphi+\partial_y$. 
In particular, $\EE_b$ induces a contact structure on $\{\varphi=\varphi_0, 0<r\le r_0\}$. The characteristic foliation on $T^2=T^2\times\{r=r_0,\varphi=\varphi_0\}$ is linear and the slope is $1/r^2>0$. 

Therefore, we can choose $r_0$ so that the even contact structures $\EE_b$ and $\EE_1$ on $Y_s$ can be glued, and we obtain an even contact structure supported by the resulting open book on $Y_s\cup_\partial \left(\partial Y\times D^2_{r_0}\right)$. The pages of this open book are formed by fibers of $Y_s\lra S^1$ with $\{\varphi=\varphi_0, 0<r\le r_0\}$ attached  along $\partial Y$. 

We now prove the uniqueness part of the Theorem. Let $\EE,\EE'$ be two even contact structures carried by the open book decomposition $(Y,\xi,h)$.
Their isotropic foliations $\WW$ and $\WW'$ are spanned by vector fields of the form $\partial_\varphi + L$ and $\partial_\varphi + L'$ on both  $N(K)$ and $Y_s$,
where $L$ and $L'$ are tangent to the pages and to $K$. 

The path of hyperplanes fields $\EE_s = \xi \oplus \R(\partial_\varphi +(1-s)L+sL')$, $s\in [0,1]$ interpolates between $\EE$ and $\EE'$, satisfies the even contact condition and is carried by $(Y,\xi ,h)$.
\end{proof}

\begin{rmk}
Note that we are able to prove that the built even contact structure is unique (cfr. Remark \ref{rem:uni}) because we start with an open book given by a triple $(Y, \xi, h)$ and not by a pair $(K, \theta)$.
\end{rmk}

Using the flexibility properties of overtwisted contact structures \cite{overtwisted}, we can show the existence of a contact structure satisfying the assumptions of Theorem~\ref{t: suspension} for a given open book decomposition $(Y,h)$ such that $\partial Y$ is a union of tori and $h$ preserves a plane field up to homotopy relative to the boundary. 

\begin{thm} \label{t: suspension overtwisted} 
Let $Y$ be a compact $3$-manifold bounded by a non-empty union of tori, and 
$h$ a diffeomorphism of $Y$ which is the identity near the boundary such that there exists a plane field $P$ of $Y$ so that $h_* P$ is homotopic to $P$ relatively to $\partial Y$.

Then the relative mapping torus of $h$ carries an even contact structure.
\end{thm} 

The proof of Theorem~\ref{t: suspension overtwisted} follows from a combination of Theorem~\ref{t: suspension} and of a variant of Proposition~\ref{p:fromframings to contact diffeos}. Under the hypothesis that the monodromy $h$ preserves a plane field $\PP$ up to homotopy relative to the  boundary, a slightly modified version of Proposition~\ref{p:fromframings to contact diffeos} gives an overtwisted contact structure homotopic to $\PP$ and preserved by $h$.
Theorem~\ref{t: suspension} then gives the conclusion.

The following theorem shows that one can always arrange that the page of an open book supporting an even contact structure is overtwisted if one allows homotopies through even contact structures which are not necessarily supported by the open book decomposition. 
In particular, $\xi'$ can be taken overtwisted even if $\xi$ is tight.
\begin{thm} \label{t: sad} 
If $\EE$ is carried by $(Y,h)$ and $\EE$ induces a contact structure $\xi$ on $Y$. Then for every overtwisted contact structure $\xi'$ on $Y$ which is homotopic to $\xi$ as oriented plane field relative to the boundary the even contact structures $\EE,\EE'$ associated to $\xi,\xi'$ are homotopic through even contact structures. 

In particular, $\xi'$ can be taken overtwisted even if $\xi$ is tight.
\end{thm} 

\begin{proof} 
It is enough to prove the result when $\xi'$ is overtwisted. In this situation, $h_* \xi'$ is overtwisted and homotopic to $\xi'$ rel. $\partial Y$. By Eliashberg's classification of overtwisted contact structures \cite{overtwisted}, $\xi'$ is isotopic to $h_* \xi'$.

We can thus deform $h$ to $h'$ to make $\xi'$ an $h'$-invariant contact structure. Then we construct an even contact structure $\EE'$ carried by $(Y,\xi',h')$. The pair $(\EE',\WW')$ is homotopic to $(\EE,\WW)$ by construction. The h-principle \cite{mcduff} implies that $\EE$ and $\EE'$ are homotopic through even contact structures: On the interior of the suspension of $(Y,h)$, the homotopy of plane fields $\zeta_t, t\in[0,1],$  interpolating between $\xi$ and $\xi'$ induces a homotopy of formal even contact structures $(\zeta_t\oplus\R\partial\varphi,\R\partial\varphi)$,  and both $\WW,\WW'$ are homotopic to $\R\partial_\varphi$ as line fields transverse to the pages.
\end{proof}

Still, rigidity could come from tightness of the pages.
\begin{q}\label{q: tight-even}
  Which even contact structures are carried by at least one open book decomposition with a tight page?
  \end{q}

\section{Open book decompositions and Engel structures}
Let $M$ be a closed oriented $4$-manifold and $(K,\theta)$ an open book decomposition of $M$.
\begin{defn} \label{d:adapted Engel}
An Engel structure $\DD$ is {\em adapted to} (or {\em supported by}) $(K,\theta)$ if the associated even
contact structure $\E =[\DD ,\DD ]$ is adapted to $(K,\theta )$.
\end{defn}

\begin{example} \label{ex:perturb prolong} 
Let $\xi$ be a contact structure on a compact $3$-manifold $Y$ which is trivial as a plane field and is supported by an open book $(K,\theta)$. We fix an adapted contact form $\alpha$ together with the corresponding Reeb vector field $R$ and a trivialization $C_1,C_2$ of $\xi$. Let $\mathrm{pr}:M= S^1 \times Y \lra Y$ be the projection. We denote the horizontal lifts of the trivialization $C_1,C_2$ and of the Reeb vector field $R$ to $S^1 \times Y$ with the same letters, and the coordinate on $S^1$ will be $t$.

For $\varepsilon>0$ small enough the plane field $\DD_k, k\ge 1$, spanned by 
\begin{align} \label{e:contact times circle}
W&=\partial_t+\varepsilon R & X&=\cos(2\pi kt)C_1+\sin(2\pi kt)C_2
\end{align}  
is an Engel structure supported by $(S^1 \times K,\theta\circ\mathrm{pr})$.
\end{example}

This example can be applied to the contact structure on $S(T^*T^2)$ for different $k$. The homotopy class of the induced framing of $S(T^*T^2)\times S^1$ depends on the parity of $k$. Thus, an open book can generate Engel structures which cannot be homotopic through Engel structures.

\subsection{Invariants of Engel structures supported by open books} 

Fix $\DD$ an Engel structure carried by $(Y,h)$, associate an integer $k$ to $\DD$ in the following way. Pick a page $Y_0$ and a collar neighborhood $\partial Y_0 \times (-1,0]$  of $\partial Y_0$ in $Y_0$ such that the characteristic foliation $\xi_0 (\partial Y_0 \times \{ s\} ), s\in (-1,0]$, is linear and where the monodromy $h$ is the identity. Let $x_0 \in \partial Y_0 \times \{ s_0 \}$ be a point in that neighborhood.

Along the path $x_0 \times S^1 \subset \Sigma_0 (Y,h)$, we have a trivialization of $\xi_t =\E \cap TY_t$ at the point $(x_0 ,t), t\in S^1$, by the line field  $C = \xi_t \cap T(\partial Y_t \times \{ s_0\} )$. There is another line field $C' =\DD \cap \xi_t$. We define $k$ to be the rotation number of $C'$ with respect to $C$ in $\xi_t$ when we go around $x_0 \times S^1$.

Now, if $\DD$ and $\DD'$ are two Engel structures  which are carried by the same open book and induce the same contact structure $\xi_0$ on one page $Y_0$, then there are two Legendrian line fields $L =\xi_0 \cap \DD$ and  $L' =\xi_0 \cap \DD'$. The degree of $L'$ with respect to $L$ induces a map which defines a homomorphism 
$$
\delta (\DD, \DD' ) : H_1 (Y_0 ,\Z ) \rightarrow \Z.
$$

If $\WW$ and $\WW'$ are the respective kernels of $\DD$ and $\DD'$ and are equal near $K$, their first return maps  $\phi$ and $\phi'$  of the kernel foliations   are contactomorphisms on the interior of $(Y_0 ,\xi_0 )$ which coincide near $\partial Y_0$. The even contact structures $\EE$ and $\EE'$ are homotopic through even contact structures supported by the open book if and only if there is a family $(\phi_t,\xi_t)$ such that $\phi_t$ is a contactomorphism of the contact structure $\xi_t$ on $Y_0$ such that $\phi_0=\phi$, $\phi_1=\phi'$, $\phi_t$ has support in the interior of $Y_0$ and $\xi_t$ is constant on a neighborhood of $\partial Y_0$.


If we moreover assume $\WW =\WW'$, then $\xi_t =\xi_t'$ for all $t\in S^1$, and $\delta (\DD ,\DD' )\circ \phi_* =\delta (\DD ,\DD' )$.


We now define {\em twisting numbers} which will be used to keep track of the homotopy class of plane fields in even contact structures such that the plane field is tangent to the isotropic foliation of the even contact structure. 

The restriction $\xi_\varphi$ of $\EE$ to each a page is a contact structure. We assume that there is a section $X$ of $\EE$ which is tangent to the pages. In the situations we are going to study, $X$ will be $-\partial_r$ on collars of the boundary of $T^2\times D^2$ and of $Y_s=Y\times[0,1]/\sim$. 

Let $C_1,C_2$ be an oriented framing of $\xi_\varphi$. For a closed oriented curve $\gamma$ there are functions $g_1,g_2$ such that $X(\gamma(t))=g_1(t)C_1(\gamma(t))+g_2(t)C_2(\gamma(t))$. We define the twisting number $\tw(C_1,\gamma)$ to be the degree of the map 
\begin{align*}
S^1 & \lra S^1 \\
t & \lmt \frac{(g_1(t),g_2(t))}{\|(g_1(t),g_2(t))\|}.
\end{align*}
Obviously, $\tw(C_1,\gamma)$ depends only on the homology class of $\gamma$ and on the homotopy class of $C_1$ as a nowhere vanishing section of $\xi_\varphi$.

\subsection{Construction of Engel structures} \label{s:existence proof}

In this section we adapt the construction from Section~\ref{s:even open} to obtain Engel structures and prove the following theorem.

\begin{thm} \label{t:existence} 
Let $Y$ be a compact $3$-manifold bounded by a non-empty union of tori, and $h$  a diffeomorphism of $Y$ which is the identity near the boundary and for  which there exists a framing $e$ of $Y$ so that $h_* e$ is homotopic to $e$ relatively to $\partial Y$.
Then the relative mapping torus of $h$ carries an Engel structure.
\end{thm} 

We will prove this theorem in Section~\ref{s:existence proof}. Using Lemma~\ref{l:geiges trick} we obtain the following corollary. 

\begin{cor}\label{cor: contact}
If $M$ is an oriented closed parallelizable $4$-manifold, every open book decomposition with toric binding $(K, \theta )$ of $M$ carries an Engel structure.
\end{cor}

\begin{prop}
Let $\DD$ be a Engel structure carried by $(Y,h)$ and $\delta : H_1 (Y,\Z )\rightarrow \Z$ be a morphism. If $\delta \circ h_* =\delta$, then there exists en Engel structure $\DD'$ supported by $(Y,h)$, which induces the same contact structure as $\DD$ on $int (Y_0 )$ (with $\WW' =\WW$ and $\EE' =\EE$ outside a neighborhood of $K$) such that $\delta (\DD ,\DD' )=\delta$.
\end{prop}

\begin{question} Is every Engel structure homotopic to an Engel structure carried by an open book?
\end{question}

\begin{question}\label{q: tight-engel} Which Engel structures are carried by an open book decomposition with tight pages?
\end{question}

Let $\EE$ be an even contact structure on the total space of a fibration over $S^1$ such that the isotropic foliation is transverse to the fibers of $\theta$. (We are going to consider two cases $\theta: Y_s\to S^1$ and  $\theta: T^2\times (D^2\setminus\{0\})\to S^1$). We denote the vertical tangent vectors in $\EE$ by $\EE^v$. The analogous notation $\DD^v$ will be used later for Engel structures $\DD$ whose isotropic foliation is transverse to a fibration over $S^1$.

In what follows,  $\gamma_x,\gamma_y,\gamma_\varphi$ are simple closed curves in  $\partial Y\times S^1$ (or, equivalently, in $T^2\times \partial D^2\subset T^2\times D^2$) which correspond to the coordinates $x,y,\varphi$. 

\subsubsection{Engel structure on $Y_s$}

Lt $S_\xi$ be vector field directing the characteristic foliation on $T^2 \times \{ r\} \times \{ \varphi\}$ such that $\partial_r,S_\xi$ is an oriented framing of $\xi$ on a neighborhood of $\partial Y$ and $\partial_r$ points into $Y_s$ along $\partial Y_s$.  We choose coordinates $x,y$ on $T^2=\partial Y$ such that 
\begin{align} \label{e:x}
\tw(C_1,\gamma_x) & =0  & \tw(C_1,\gamma_y) &=\lambda
\end{align}
where $\lambda$ is an integer. This determines the homotopy class of $C_1$ as a nowhere vanishing section of $\xi$ and we assume
\begin{align} \label{e:C1C2}
\begin{split}
C_1 & = \cos(\lambda y)\partial_r+\sin(\lambda y)S_\xi\\
C_2 & = -\sin(\lambda y)\partial_r+\cos(\lambda y)S_\xi
\end{split}
\end{align}
on a collar of $T^2\subset\partial Y$.  A simple computation shows that $C_1$ has the right twisting numbers (cf. \eqref{e:x}). Note in particular, that along $\{ y=0\}$, $C_1 =-\partial_r$ and $C_2 =S_\xi$. 
 
On $Y\times[0,2]$ with coordinates $(p, \varphi)$ consider the plane field $\DD_k$ spanned by the isotropic foliation of the even contact structure defined as in Section \ref{s:even open} using the contact structure $\xi$, and 
\begin{equation}\label{e:Xk}
X_k:=\cos(2\pi k\varphi)C_1+\sin(2\pi k\varphi)C_2
\end{equation}
 with $k\in\mathbb{N}$. When $k>0$ is large enough $\DD_k$ is an Engel structure and $[\DD_k,\DD_k]=\EE$ is independent of $k$. The orientation of $\EE$ induced by $\DD_k$ is the one determined by the framing $C_1,C_2$. Moreover, assuming again that $k$ is large enough, there is a smooth function $f : Y\lra(0,2)$ with $f\equiv 1$ on $U$ (where $h=id$) such that 
\begin{equation} \label{e:untwist framing}
h : Y\simeq \{(p,f(p))|p\in Y\}  \lra Y=Y\times\{0\}
\end{equation}
preserves the line field spanned by $X_k$. At this point we use the assumption that the framing $(h_*(C_1),h_*(C_2))$ is homotopic to $(C_1,C_2)$ together with the fact that the holonomy of the characteristic foliation of $\DD_k$ is the identity on $U$. Thus, we obtain an Engel structure $\DD_k$ on $Y_s$.
Notice that near $\partial Y_s$ the Engel structure $\DD_k$ is  $\langle \partial_\varphi +\partial_y \rangle\oplus \langle X_k\rangle$.

By construction $\tw(X_k,\gamma_\varphi)=k$
. In order to make subsequent constructions possible, we require that $k$ is odd and positive.

\subsubsection{Even contact structure on $T^2\times D^2$}

We need to extend the Engel structure $\DD_k$ to $T^2\times D^2$. For that we will first extend the associated even contact structure $\EE$ in such a way that the plane field $\DD_k$ also  extends to a plane field in $\EE$ containing $\WW$. This extra condition requires a modification of the construction in Section~\ref{s:even open}.

We will choose the even contact structure on $T^2\times D^2=S_{x}^1\times S_{y}^1\times D^2$ such that it has a section $C_1'$ extending $X_k$ over $(0,0)\times D^2$ which is never tangent to $\WW$. Note that the homotopy type of $X_k$ as a section of $\EE^v$ along $(0,0)\times \partial D^2$ depends only on $k$ (and not on the choice of $C_1$ in $Y$ since it is $\varphi$-invariant). This is done as follows.

On $S^1_{x}\times D^2$ consider an $x$-invariant contact structure $\xi'$ with the following properties:
\begin{itemize}
\item $\xi'=\ker(dx+r^2d\varphi)$ near the boundary of $S^1_{x}\times D^2$. In particular, $\xi' =\xi$ along $\partial Y_s =\partial N(K)$.
\item The characteristic foliation on $\{x=0\}$ has only non-degenerate singularities. Let $e_\pm$, respectively $h_\pm$, denote the number of elliptic, respectively hyperbolic, singularities. The subscript refers to the sign of the singularities.  We require 
\begin{align} \label{e:sing for extension}
\begin{split}
e_+=(k+1)/2  & \quad e_-=0\\ 
h_+=0 & \quad h_-=(k-1)/2.
\end{split}
\end{align}
Since $k$ is odd, these numbers are all integers. 
\end{itemize}
Such a foliation exists since $e_++e_--h_+-h_-=1=\chi(D^2)$ and standard results from the theory of convex surfaces imply that the singular foliation on $\{x=0\}$ determines a $x$-invariant contact structure $\xi'$ on $S^1_{x}\times D^2$.
We consider the extension $\EE'$ of $\EE$ of given by $\EE' =\xi' \oplus \langle \partial_y +\partial_\varphi \rangle$.

We let $X'_k$ be a vector field along $\partial Y_s$ generating $\DD_k \cap \xi'$. Note that 
$$
e(\xi')[D^2] = e_+-e_--h_++h_-= k = \tw(X_k', \partial D^2). 
$$
By the Poincar{\'e}-Hopf theorem, the Legendrian vector field $X'_k$  of $\xi'|_{\partial D^2}$ extends to a nowhere vanishing section $C_1'$ of $\EE'$ on $T^2 \times (0,0)$, everywhere transverse to $\WW' = \partial_y +\partial_\varphi$.

This extension is unique up to homotopy and there is a Legendrian vector field $C_2'$ such that $(C_1',C_2')$ is a framing of $\xi'$ on $ (0,0)\times D^2 $. 
 
Near $T^2\times \partial D^2$ we require 
$C_i$ and $C_i'$ have the same image in $\EE/\WW$  for $i=1,2$. Then, there is a homotopy between $C_i$ and $C_i'$ in $\DD_k \setminus \WW$ so that $C_i'$ can be viewed as an extension of $C_i$. Near $T^2\times \partial D^2$, we have:

\begin{align}\label{e:Y1Y2}
\begin{split}
C_1' & = \cos(k\varphi)\partial_r+\sin(k\varphi)S_\xi \\
C_2' & = -\sin(k\varphi)\partial_r+\cos(k\varphi)S_\xi.
\end{split}
\end{align}
This is consistent with the choice of $X_k$ near $\partial Y_s$, i.e. the $C_i$ together with $C_i'$ forms a smooth vector field on a neighborhood of $T^2\times\partial D^2$ in  $M=Y_s\cup_\partial (T^2\times D^2)$ for $i=1,2$. 
 
\subsubsection{Engel structure on $T^2\times D^2$}

Now we define $\tilde{\DD}_l$ on $T^2\times D^2$ as the span of $\partial_{y}$ and $\tilde{X}_l=\cos(ly)C_1'+\sin(ly)C_2'$. This is an Engel structure for all $l\in\mathbb{N}$. The isotropic foliation is spanned by $\partial_{y}$, so it is not transverse to the fibers of $\theta$. However, the Engel structure $\DD'_l$ obtained by pushing forward $\tilde{\DD}_l$ with the diffeomorphism 
\begin{align*}
f : S^1\times(S^1\times D^2) & \longrightarrow S^1\times(S^1\times D^2) \\
(x,y,r,\varphi) & \lmt (x,y,r,\varphi+y)
\end{align*} 
has the desired property: Its isotropic foliation is spanned by $\partial_{y}+\partial_{\varphi}$, so it is transverse to the fibers of $\theta$.  Let 
\begin{equation} \label{e:Xl}
X_l'=f_*(\tilde{X}_l)= \cos(ly)f_*(C_1')+\sin(ly)f_*(C_2').
\end{equation}
By construction, the gluing map $g$ maps the isotropic foliation (which is tangent to $\partial Y_s$) to the isotropic foliation of $\DD_l'$ (which is also tangent to $T^2\times \partial D^2$).  As in Section \ref{s:even open} we can arrange that the gluing map preserves the singular foliation on the fibers of $\theta$.

\subsubsection{Gluing the pieces together}

The even contact structures induced by $\DD_k$, respectively $\DD_l'$, coincides with the even contact structures near $\partial Y\times S^1$, respectively $T^2\times\partial D^2$, which arose  in Section \ref{s:even open}. Therefore, the only remaining problem is to ensure that the gluing maps the line field $\DD_l'^v$ to  $\DD_k^v$.  From \eqref{e:Xk} and \eqref{e:C1C2} we obtain
\begin{align*}
X_k 
  & = \cos(k\varphi+\lambda y)\partial_r+\sin(k\varphi+\lambda y)S_\xi \\
\intertext{and, according to \eqref{e:Y1Y2} and \eqref{e:Xl}, $X_l'$ is given by }
X_l' 
       & = \cos((l-k)y+k\varphi)\partial_r+\sin((l-k)y+k\varphi)S_\xi
\end{align*}
near $T^2\times \partial D^2$. The conditions on the parameters $l,k$ are
\begin{itemize}
\item $k$ is a positive odd integer,
\item $l$ is a positive integer.
\end{itemize}
Thus, we obtain a smooth Engel structure on $M=(T^2\times D^2)\cup_\partial Y_s$ if $l-k=\lambda$.  

\subsubsection{Uniqueness}
Using recent results about flexibility of Engel structures \cite{cpp} we can prove the following

\begin{thm} \label{t:loose unique}
For a fixed choice of framing $e$, the Engel structures constructed in Theorem \ref{t:existence} are unique up to homotopy through Engel structures for $k$ large enough.
\end{thm}

According to \cite{cpp}, we only need to check that 
\begin{itemize}
\item any pair of them are homotopic through formal Engel structures, i.e. the induced framings of $TM$ are homotopic, and 
\item the Engel structure is loose in the sense of \cite{cpp} when $k$ is big enough. 
\end{itemize} 
Recall that $k$ is odd, therefore the Engel structures on the mapping torus associated to $k$ and $k+2$ are formally homotopic, i.e. the associated framings are homotopic. Near the boundary, we can be more specific using the coordinates $x,y$ and $\varphi$ on the boundary of the mapping torus: The homotopy of framings can be chosen to be 
\begin{itemize}
\item constant on the neighborhod $D_{1/2}\times  T^2$ of the binding and we may assume that the characteristic foliation on $D_{1/2}$ contains all singular points introduced during the construction \eqref{e:sing for extension} for a fixed, odd value of $k$ at the beginning of the homotopy. At the end of the homotopy, $D_1\setminus\ring{D}_{1/2}$ contains an additional pair of singularities so that at the end of the homotopy, the characteristic foliation on $D_1$ has the number of singular points prescribed by \eqref{e:sing for extension} for $k+2$ instead of $k$.
\item the hyperplane field on $D_1\setminus D_{1/2}$ spanned by the first three components of the framing is invariant under rotations in the $x,y$-direction throughout the homotopy,
\item throughout the homotopy, the framing is invariant under translations in the $x$-direction, the first and the last component of the framing are  also $y$-invariant. 
\item  The second and the third components are not $y$-invariant in general, but they are twisted and the twisting is determined by $\lambda$ as in \eqref{e:C1C2}. 
\end{itemize}
A homotopy with all these properties can be found by first fixing it on the annulus $D_1\setminus \ring{D}_{1/2}$ and then extending the homotopy using the desired invariance properties. 

Since $\pi_2(\mathrm{SO}(4))=\{0\}$, this homotopy is unique and we may assume that the at the end of the homotopy, the characteristic foliation on $D_1\setminus\ring{D}_{1/2}\times \{(0,0)\}\subset D^2\times T^2$ has precisely one elliptic and one hyperbolic singularity (as in \eqref{e:C1C2} for $k+2$ instead of $k$).

The end result of the homotopy of framings coincides (up to homotopy) with the Engel framing obtained by our construction for $k+2$. In particular, the twisting number along $\gamma_y$ is fixed.

As for the looseness, we use the following definition. A chart in $(M, \DD)$ is an $N$-Darboux chart, $N\in \mathbb{N}$, if there is an Engel embedding of the ball $D^3\times (0, N)$ with coordinates $(x,y,z, \theta)$ and with Engel structure $(\ker (dy-zdx), \ker (\cos (2\pi (\theta-\theta_0)) dz- \sin(2\pi (\theta-\theta_0))dx$ for some $\theta_0\in \R$. 

In other words, the chart is a Darboux chart of an even contact structure in which the vector transverse to $\WW$ rotates $N$ times. We say that $\DD$ is $N$-loose if any point of the manifold can be placed in the interior (not necessarily in the middle) of a  $N$-Darboux chart. The main result of \cite{cpp} implies that there is a integer $N_0$ with the following property: If $\DD_0$ and $\DD_1$ are $N$-loose with $N>N_0$ and formally homotopic, then they are homotopic through Engel structures. 

The Engel structure on $Y_s$ is $k-L$-loose, where $L$ depends only on the gluing map $\varphi$ and the framing (see \eqref{e:untwist framing}). The Engel structure $\DD'_l$ on $T^2\times D^2$ is $l$-loose. Finally, recall that $l=k+\lambda$. Since both $\lambda$ and $L$ are independent of $k$, this implies Theorem~\ref{t:loose unique}.

\section{Stabilization}

Let $(Y,\xi )$ be a contact manifold whose boundary is a non-empty union of tori so that the characteristic foliation of $\xi$ on each component of  $\partial Y$ is linear. 

Let $A$ be a  properly embedded annulus in $(Y,\partial Y)$ such that
\begin{itemize}
\item[(i)] $\partial A$ is not contractible in $\partial Y$,
\item[(ii)] the characteristic foliation $\xi A$ is non-singular and made of intervals going from one boundary component to the other, and
\item[(iii)] the characteristic foliations on the two $2$-tori formed by  $A$ and the two annuli cut by $\partial A$ in $\partial Y$ are linear.
\end{itemize}

Like in the topological case, we glue a handle $S^1 \times [0,1] \times [-1,1]$
to $\partial Y$ and consider the torus $T =A \cup S^1 \times [0,1] \times \{ 0\}$. Recall that the stabilization procedure requires the choice of a curve $\gamma$ in $T$ intersecting the co-core $S^1 \times \{ 1/2\} \times \{ 0\}$ exactly once.

We extend $\xi$ to the handle $S^1 \times [0,1] \times [-1,1]$ in the following way.
First, we  extend the characteristic foliation $\xi A$ of $A$ to a linear foliation of $T$ such that $\gamma$ is an integral curve. This  determines the germ of the contact structure $\xi'$ along $T$. After reducing the thickness of the handle, we obtain a contact structure $\xi'$ on $Y' =Y \cup S^1 \times [0,1] \times [-1,1]$. Because of condition (iii), the characteristic foliations of the two tori formed by gluing $S^1 \times [0,1] \times \{ 0\}$ and the two annuli cut by $\partial A= S^1 \times \{0,1\} \times \{ 0\}$ in $\partial Y$ are linear. This implies that we can choose the boundary of $Y'$ so that the  characteristic foliation of $\xi'$ is linear.

There exists a neighborhood $T\times [-\varepsilon ,\varepsilon ]$ of $T\simeq T\times \{ 0\}$ such that each torus $T\times \{ s\}$, $s\in [-\varepsilon ,\varepsilon ],$ has a linear characteristic foliation for $\xi'$. We now choose a representative of $\tau_\gamma$ which preserves 
each torus together with its foliation. Following Giroux \cite{giroux, etnyre}, we then obtain another representative $\tau_\gamma$ which preserves $\xi'$. Let $h' =\tau_\gamma \circ h$.

Continuing Example~\ref{ex:open book even} we give an example of a situation where the above conditions are satisfied.
\begin{example} \label{ex:A prolongation}
Let $(K,\theta)$ be an open book decomposition of a closed $3$-manifold $Y$ supporting a contact structure $\xi$ which is trivial as a vector bundle. In Example~\ref{ex:open book even} we showed that for $\varepsilon>0$ small enough there is an even contact structure $\EE_\varepsilon$ on $S^1 \times Y$ which is homotopic to $\mathrm{pr}_*^{-1}(\xi)$ through even contact structures. Recall that $\mathrm{pr}$ denotes the projection onto the second factor of $S^1 \times Y$.
We now consider a stabilization of $(K,\theta)$. For this we choose a properly embedded arc $a$ inside a page $S$. By Giroux's realization Lemma \cite{girouxr}, it is always possible to isotope $\xi$ through contact structures carried by $(K,\theta)$ so that $a$ is Legendrian. 
Then the open book decomposition $(\mathrm{pr}^{-1}(K),\theta\circ\mathrm{pr})$ is adapted to $\EE_\varepsilon$ and the annulus $S^1 \times a$ verifies the conditions stated at the beginning of this section. 
\end{example}

\begin{thm}\label{thm: stabilization even} The even contact structures $\EE$ and $\EE'$ determined by the mapping tori of $(Y,\xi ,h)$ and $(Y',\xi' ,h')$ are homotopic.
\end{thm}

\begin{proof}
Recall first that if an open book decomposition supporting a contact structure  $\xi$ on a closed $3$-manifold $Z$  is stabilized along an arc $a$ properly embedded in a page, then the stabilized open book decomposition also carries a contact structure $\xi'$ and that $\xi$ and $\xi'$ are isotopic. More precisely, after Giroux-Goodman \cite{giroux-goodman} (also see \cite{etnyre}), one chooses  $a$ to be Legendrian and then
$$(Z,\xi')=(Z,\xi)\#(S^3,\xi_{st}) $$ 
where the sphere separating the two summands bounds a small neighborhood of $a$ in $Z$. The connected sum with $S^3$ does of course not change the manifold and the contact structure $\xi'$ is isotopic to $\xi$ since one Darboux ball (the neighborhood $N(a)$ of $a$) is removed and another Darboux ball is glued in. But then $\xi'$ is homotopic to $\xi$ relative to the complement of $N(a)$ since all tight contact structures on a ball which induce the same characteristic foliation on the boundary are isotopic by Eliashberg \cite{El2} (although in this case it can be checked directly).

Now let $A\simeq S^1\times[0,1]$ be an annulus satisfying the conditions (i), (ii) and (iii) stated at the beginning of this section. Then one can find a neighborhood $N_0(A)\simeq A\times (-\delta,\delta)$ of $A$ in the page $Y$ so that the plane field $\xi$ is defined by $dt-f(y)dx$, where $(t,x)$ are product coordinates on $A=S^1\times[0,1]$ and $y\in(-\delta,\delta)$. Remember that $\xi$ is contact in the interior of $Y$, so $f'(y)>0$ when $x\in (0,1)$ and $f'(y)=0$ on the binding $\partial Y\cap N_0(A)=\{x=0,1\}$.

In order to describe the even contact structure $\EE'$ one can now refer to Example~\ref{ex:A prolongation} when we view $A=S^1\times a$ where $a$ is a leaf of the characteristic foliation of a fixed annular page of the open book of $S^3$ described in Example~\ref{ex:S3 open book}, joining the two components of the Hopf link binding. In particular, there is a neighborhood $N(A)$ of $N_0(A)$ in $M$ where $\EE$ is, possibly after homotopy, conjugate to $(S^1\times N(a),\EE_\varepsilon)\subset (S^1 \times S^3 ,\EE_\varepsilon)$.

Notice that the contact structure $\xi_{st}$ on $S^3$ supported by this open book is not trivial as a vector bundle, but it is over trivial $N(a)$ where we can apply the construction of Example~\ref{ex:open book even}. The even contact structure $\EE'$ is obtained from $\EE$ by first removing $(S^1\times N(a),\EE)$ and then gluing in $(S^1\times(S^3\setminus N(a)),\EE_\varepsilon)$. Here again, $\xi_{st}$ is trivial over $S^3\setminus N(a)$, so the construction makes sense, and there is an orientation reversing identification sending the boundary of $(S^1\times N(a),\EE_\varepsilon)$ to the boundary of $(S^1\times(S^3\setminus N(a)),\EE_\varepsilon)$.
It now follows in particular that $(\EE',\WW')$ and $(\EE,\WW)$ are homotopic as formal even contact structures, and hence as even contact structures \cite{mcduff}. 
\end{proof}

In view of questions \ref{q: tight-even} and \ref{q: tight-engel}, a nice feature of the stabilization operation is the following:

\begin{thm}\label{tight}
If $(Y,\xi )$ is universally tight then so is $(Y',\xi')$.
\end{thm}
\begin{proof} The contact manifold $(Y',\xi')$ is obtained by gluing a universally tight contact structure on $T^2 \times [0,1]$ to $(Y,\xi)$ along the prelagrangian torus $\partial Y\subset (Y,\xi)$. The result is universally tight by \cite{Co}.
\end{proof} 

Finally, we describe how to obtain an Engel structure $\DD'$ through a stabilization procedure on an open book decomposition supporting an Engel structure $\DD$ on a $4$-manifold $M$.
In many situations, as in the case of even contact structures, if an open book $(Y,\xi,h)$ supports an Engel structure $\DD$ and if $A\subset (Y,\xi)$ is an annulus satisfying conditions $i),ii)$ and $iii)$, then one can construct an Engel structure $\DD'$ carried by the stabilization $(Y',\xi',h')$ of $(Y,\xi,h)$ along $A$, by the use of Theorem \ref{t:existence}. We can check that they are formally homotopic and by \cite{cpp} they are homotopic for $k$ large enough, since we have shown that they are loose. However, we lose control of the geometry. 

We present below a more intrinsic operation.
Unfortunately this construction involves a specific rigid normal form for $\DD$ near the annulus $A$ along which we stabilize that we do not know how restrictive it is.

More explicitely, in the Example \ref{ex:perturb prolong} where $\xi_v$ is a contact structure on some closed $3$-manifold $Y$ supported by an open book $(K,\theta)$, we pick a Legendrian arc $a$ properly embedded in a page $S$ of $(K,\theta)$ and consider the annulus $S^1 \times a =\mathrm{pr}^{-1} (a)$ together with an $S^1$-invariant neighborhood of it $S^1 \times N(a)$. 
We assume that the Engel structure $\DD$ on a neighborhood $N(A)$ of $A$ is conjugated with some Engel structure $\DD_k$ from Example \ref{ex:perturb prolong} on $S^1 \times N(a)$. Recall the construction of $\DD_k$ involves the choice of a trivialization and of a $(K,\theta)$ adapted Reeb vector field of $\xi_v$ on $N(a)$.

In order to obtain an Engel structure carried by the open book of $M$ stabilized along $A$, we replace the Engel structure $\DD \simeq \DD_k$ on $N(A)\simeq S^1\times N(a)$  with an Engel structure $\DD'$ obtained in the following way.
We stabilize the open book $(K,\theta)$ of $Y$ along $a$ using Giroux and Goodman's \cite{giroux-goodman} stabilization construction. This is given, as mentioned in the proof of Theorem~\ref{thm: stabilization even}, by replacing $\xi_v$ and the partial open book  induced by $(K,\theta)$ on $N(a)$ by a structure $\xi_v'$ -- contactomorphic to $\xi_v$ -- and a stabilized partial open book $(K',\theta')$ supporting $\xi_v'$ on $N(a)$. The trivialization and the Reeb vector field of $\xi_v$ along $\partial N(a)$ extend to a trivialization and a $(K',\theta')$ adapted Reeb vector field of $\xi_v'\simeq \xi_v$ on $N(a)$.
We then apply again the construction  of  Example~\ref{ex:perturb prolong} to $\xi_v'$ and $(K',\theta')$ to  get $\DD'$ as an Engel structure $\DD_k'$ on $N(A)\simeq S^1 \times N(a)$. It is supported by the original open book of $M$ stabilized along $A$.

This concludes the description of $\DD'$. As in the case of even contact structures it follows that $\DD'$ is homotopic to $\DD$ through Engel structures: The space of framings of $\xi'_v$ which coincide with a fixed framing on $\partial N(a)$ is contractible. Moreover, notice that $\xi_v$ and $\xi_v'$ are both tight near $\partial N(a)$. According to a theorem of Eliashberg \cite{El2}, any two contact structures on the ball which coincide near $\partial N(a)$ are isotopic relative to the boundary of the ball.  Finally, any two contact vector fields which are transverse to a fixed contact structure are homotopic through contact vector fields which are transverse to that contact structure (this can be seen using convex combinations of two such contact vector fields). 

\begin{rmk}
Given an annulus $A$ in a page $(Y,\xi)$ satisfying conditions $i),ii)$ and $iii)$, it is not clear whether one can find a neighborhood $N(A)$ of $A$ and a homotopy of $\DD$ where $\DD$ becomes conjugated to one of our models. 
\end{rmk}

\end{document}